\documentclass{amsart}
\usepackage{amssymb, amscd, enumerate} 

\newtheorem{theorem}{Theorem}[section]
\newtheorem{lemma}[theorem]{Lemma}

\newtheorem{cor}[theorem]{Corollary}
\newtheorem{prop}[theorem]{Proposition}

\theoremstyle{definition}
\newtheorem{definition}[theorem]{Definition}

\theoremstyle{remark}
\newtheorem{remark}[theorem]{Remark}

\numberwithin{equation}{section}

\newcommand{\al}{\alpha}
\newcommand{\be}{\beta}
\newcommand{\ga}{\gamma}
\newcommand{\Ga}{\Gamma}

\newcommand{\ep}{\varepsilon}

\newcommand{\la}{\lambda}

\newcommand{\si}{\sigma}
\newcommand{\Si}{\Sigma}

\newcommand{\csi}{\xi}

\newcommand{\x}{\times}

\newcommand{\CC}{\mathcal C}

\newcommand{\imm}{{\mathrm {Imm}}}

\newcommand{\Z}{\mathbb Z}

\newcommand{\Q}{\mathbb Q}
\newcommand{\R}{\mathbb R}

\newcommand{\RP}{{\mathbb R}{P}}

\newcommand{\del}{\partial}

\newcommand{\co}{\colon\thinspace}

\begin{document}
\title{Cobordism Invariants of Fold Maps}

\author{Boldizs\'ar Kalm\'{a}r}

\address{Kyushu University, Faculty of Mathematics, 6-10-1 Hakozaki, Higashi-ku, Fukuoka 812-8581, Japan}
\email{kalmbold@yahoo.com}


\thanks{The author was supported by Canon Foundation in Europe.}

\subjclass[2000]{Primary 57R45; Secondary 57R75, 57R42, 55Q45}

\dedicatory{Dedicated to Professor Yoshifumi Ando on the occasion of his sixtieth birthday.}

\keywords{Fold singularity, fold map, immersion, cobordism, stable homotopy group}


\begin{abstract}
This is a survey paper of author's results on cobordism groups and semigroups
of fold maps and simple fold maps. The results include:
establishing a relation between fold maps and immersions through geometrical
invariants of cobordism classes of fold maps and simple fold maps in terms of immersions
with prescribed normal bundles, 
detecting stable homotopy groups of spheres as direct summands of the cobordism semigroups of fold maps,
Pontryagin-Thom type construction for $-1$ codimensional fold maps and  
estimations about the cobordism classes
of manifolds which have fold maps into stably parallelizable manifolds.
In the last section some of these results are extended and
we show that our invariants also detect stable homotopy groups of the classifying spaces $BO(k)$ as direct 
summands of the cobordism semigroups of fold maps.
\end{abstract}

\maketitle

\section*{Introduction}

Fold maps of $(n+q)$-dimensional manifolds into $n$-dimensional manifolds have the formula 
\[
f(x_1, \ldots, x_{n+q})=(x_1, \ldots, x_{n-1}, \pm x_n^2 \pm \cdots \pm x_{n+q}^2 )
\]
as a local form around each singular point, and  
the subset of the singular points in the source manifold is a $(q+1)$-codimensional submanifold 
(for results about fold maps, see, for example, \cite{An2, An5, An3, An, Eli1, Eli2, Kal5, Sad, Sa5, Sa}).
If we restrict a fold map to the set of its singular points, then we obtain a
codimension one immersion into the target manifold of the fold map.
This immersion together with more detailed informations about the neighbourhood of the set of
singular points in the source manifold can be used as a geometrical invariant
 (see Section~\ref{foldgerms})
of fold cobordism classes (see Definition~\ref{cobdef}) of fold maps 
(for results about cobordisms of singular maps with completely
different approach from our present paper, see, for example, 
\cite{
EkSzuTer, Ik, IS, 
Ko,
RSz, 
Saspecgen, We} and the works of Ando, Sadykov, Sz\H{u}cs and the author in References). In 
this way we obtain a geometrical relation between fold maps and immersions 
with prescribed normal bundles via cobordisms.
In \cite{Kal5} we showed that these invariants describe completely
the cobordisms of simple fold maps of $(n+1)$-dimensional manifolds into $n$-dimensional manifolds
and in \cite{Kal4} we showed that these invariants detect direct summands of
the cobordism group of fold maps, namely stable homotopy groups of spheres. In this paper
we extend the results of \cite{Kal4} and show 
by constructing fibrations of Morse functions over
immersed manifolds similarly to \cite{Kal4} that these invariants
also detect stable homotopy groups of the classifying spaces $BO(k)$ as direct summands of the
cobordism semigroups of fold maps.

The paper is organized as follows.
In Section~\ref{prelim} we give basic notations and definitions,
in Section~\ref{foldgerms} we define cobordism invariants of fold maps, summerize 
our already existing results concerning these invariants and
study the cobordism classes of manifolds which have fold maps into stably parallelizable manifolds.
In Section~\ref{subgroups} we extend the results of \cite{Kal4}.

The author would like to thank the referee for his helpful comments, which 
improved the paper.

\section{Preliminaries}\label{prelim}

\subsection{Notations}
In this paper the symbol ``$\amalg$'' denotes the disjoint union, 
for any number $x$ the symbol ``$\lfloor x \rfloor$'' denotes the greatest
integer $i$ such that $i \leq x$,
$\ga^1$ denotes the universal line bundle over $\RP^{\infty}$,
$\ep^1_X$ denotes the trivial line bundle over the space $X$,
$\ep^1$ denotes the trivial line bundle over the point,
and the symbols $\csi^k$, $\eta^k$, etc. usually denote $k$-dimensional real vector bundles.
The symbols det$\csi^k$ and $T\csi^k$ denote the determinant line bundle 
and the Thom space of the bundle $\csi^k$, respectively.
The symbol $\imm^{\csi^k}_{N}(n-k,k)$ denotes
the cobordism group of $k$-codimensional immersions into 
an $n$-dimensional manifold $N$
whose normal bundles are induced from $\csi^k$ (this
group is isomorphic to the group $\{\dot N, T\csi^k \}$, where
$\dot N$ denotes the one point compactification of the manifold $N$
and the symbol $\{X,Y \}$ denotes the group of stable homotopy classes of continuous
maps from the space $X$ to the space $Y$). 
The symbol $\imm_N(n-k,k)$ denotes
the cobordism group $\imm^{\ga^k}_N(n-k,k)$ where
$\ga^k$ is the universal bundle for $k$-dimensional real vector bundles and 
$N$ is an $n$-dimensional manifold.
The symbol $\pi_n^s(X)$ ($\pi_n^s$) denotes the $n$th stable homotopy group of the space $X$ (resp. spheres).
The symbol ``id$_A$'' denotes the identity map of the space $A$.
The symbol $\ep$ denotes a small positive number.
All manifolds and maps are smooth of class $C^{\infty}$.
 
\subsection{Fold maps}
 
Let $n \geq 1$ and $q > 0$.
Let $Q^{n+q}$ and $N^n$ be smooth manifolds of dimensions $n+q$ and $n$ 
respectively. Let $p \in Q^{n+q}$ be a singular point of 
a smooth map $f \co Q^{n+q} \to N^{n}$. The smooth map $f$  has a {\it fold 
singularity of index $\la$} at the singular point $p$ if we can write $f$ in some local coordinates around $p$  
and $f(p)$ in the form 
\[  
f(x_1,\ldots,x_{n+q})=(x_1,\ldots,x_{n-1}, -x_n^2 - \cdots -x_{n+\la-1}^2 + x_{n+\la}^2 + \cdots + x_{n+q}^2)
\] 
for some $\la$ $(0 \leq \la \leq q+1)$ (the index $\la$ is well-defined if
we consider that $\la$ and $q+1-\la$ represent the same index). 

A smooth map $f \co Q^{n+q} \to N^{n}$ is called a {\it fold map} if $f$ has only 
fold singularities.

A smooth map $f \co Q^{n+q} \to N^n$ 
  has a {\it definite fold
singularity} at a fold singularity $p \in Q^{n+q}$ if $\la = 0$ or $\la = q+1$,
otherwise $f$ has an {\it indefinite fold singularity of index $\la$}
at the fold singularity $p \in Q^{n+q}$.

Let $S_{\la}(f)$ denote the set of fold singularities of index $\la$ of $f$ in $Q^{n+q}$.
Note that $S_{\la}(f) = S_{q+1-\la}(f)$.
 Let $S_f$ denote the set $\bigcup_{\la} S_{\la}(f)$.

Note that the set $S_f$ is an ${(n-1)}$-dimensional submanifold of the manifold
$Q^{n+q}$.

Note that each connected component of the manifold $S_f$ has its own index $\la$ if
we consider that $\la$ and $q+1-\la$ represent the same index. 

Note that for a fold map $f \co Q^{n+q} \to N^{n}$ and for an index $\la$ ($0 \leq \la \leq \lfloor q/2 \rfloor$)
the codimension one immersion $f \mid_{S_{\la}(f)} \co S_{\la}(f) \to N^n$ 
of the singular set of index $\la$ $S_{\la}(f)$ has a canonical framing 
(i.e., trivialization of the normal bundle) by identifying canonically the set of 
fold singularities of index $\la$ 
of the map $f$ with   
the fold germ 
$(x_1,\ldots,x_{n+q}) \mapsto (x_1,\ldots,x_{n-1}, -x_n^2 - \cdots -x_{n+\la-1}^2 + x_{n+\la}^2 + \cdots + x_{n+q}^2)$,
see, for example, \cite{Sa1}.

If $f \co Q^{n+q} \to N^n$ is a fold map in general position, then 
the map $f$
restricted to the singular set $S_f$ is a general positional
 codimension one immersion  into the target manifold $N^n$.
 
Since every fold map is in general position after a small perturbation, 
and we study maps under the equivalence relation {\it cobordism}
(see Definition~\ref{cobdef}),
in this paper we can restrict ourselves to studying fold maps which are 
in general position.
Without mentioning we suppose that a fold map $f$ is in general position.

\subsection{Equivalence relations of fold maps}\label{kob}

\begin{definition}\label{cobdef} (Cobordism) 
Two fold maps $f_i \co Q_i^{n+q} \to N^n$ $(i=0,1)$  
of closed (oriented) $({n+q})$-dimensional manifolds $Q_i^{n+q}$ $(i=0,1)$ 
into an $n$-dimensional manifold $N^n$ are  
{\it \textup{(}oriented\textup{)} cobordant} if 
\begin{enumerate}[(a)]
\item
there exists a fold map 
$F \co X^{n+q+1} \to N^n \times [0,1]$ of a compact (oriented) $(n+q+1)$-dimensional 
manifold $X^{n+q+1}$, 
\item
$\del X^{n+q+1} = Q_0^{n+q} \amalg (-)Q_1^{n+q}$ and 
\item
${F \mid}_{Q_0^{n+q} \x [0,\ep)}=f_0 \x
{\mathrm {id}}_{[0,\ep)}$ and ${F \mid}_{Q_1^{n+q} \x (1-\ep,1]}=f_1 \x 
{\mathrm {id}}_{(1-\ep,1]}$, where 
$Q_0^{n+q} \x [0,\ep)$
 and $Q_1^{n+q} \x (1-\ep,1]$ are small collar neighbourhoods of $\del X^{n+q+1}$ with the
identifications $Q_0^{n+q} = Q_0^{n+q} \x \{0\}$ and $Q_1^{n+q} = Q_1^{n+q} \x \{1\}$. 
\end{enumerate}

We call the map $F$ a {\it cobordism} between $f_0$ and $f_1$.
\end{definition} 
This clearly defines an equivalence relation on the set of fold maps 
of closed (oriented) $({n+q})$-dimensional manifolds into an  
$n$-dimensional manifold $N^n$.

We denote 
 the set of fold (oriented) cobordism classes of fold maps of closed (oriented) $({n+q})$-dimensional manifolds 
into an $n$-dimensional manifold 
$N^n$ (into the Euclidean space $\R^n$)
by $\CC ob_{N,f}^{(O)}(n+q,-q)$ (by $\CC ob_{f}^{(O)}(n+q,-q)$).
We note that we can define a commutative semigroup operation in the usual way on the 
set of cobordism classes $\CC ob_{N,f}^{(O)}(n+q,-q)$
by the disjoint union.
If the target manifold $N^n$ is the Euclidean space $\R^n$ (or more
generally $N^n$ has the form $\R^1 \x M^{n-1}$, for some $(n-1)$-dimensional manifold  
$M^{n-1}$), then the elements in the semigroup $\CC ob_{N,f}^{(O)}(n+q,-q)$
have an inverse: namely compose them with a reflection
in a hyperplane (in $\{0\} \x M^{n-1}$ in general, see \cite{Szucs4}).
Hence the semigroups $\CC ob_{N,f}^{(O)}(n+q,-q)$ are in this case 
actually groups.

We can refine this equivalence relation by considering the 
singular fibers (see, for example, \cite{Lev, Sa, SaYa, Ya}) of a fold map.

\begin{definition}
Let $\tau$ be a set of singular fibers.
Two fold maps $f_i \co Q_i^{n+q} \to N^n$ $(i=0,1)$ with singular fibers in the set $\tau$  
of closed (oriented) $({n+q})$-dimensional manifolds $Q_i^{n+q}$ $(i=0,1)$ 
into an $n$-dimensional manifold $N^n$ are  
{\it \textup{(}oriented\textup{)} $\tau$-cobordant} if 
they are (oriented) cobordant in the sense of Definition~\ref{cobdef}
by a fold map $F \co X^{n+q+1} \to N^n \times [0,1]$ 
whose singular fibers are in the set $\tau$.
\end{definition}

In this way we can obtain the notion of {\it simple fold cobordism} of {\it simple fold maps}, i.e., let
$\tau$ be the set of all the singular fibers which have at most one singular point in each 
of their
connected components. We denote the set of simple fold cobordism classes of simple fold maps 
of closed (oriented) $({n+q})$-dimensional manifolds $Q^{n+q}$  
into an $n$-dimensional manifold $N^n$
by $\CC ob_{N,s}^{(O)}(n+q,-q)$. For results about simple fold maps, see, for example, 
\cite{Kal5, Sa1, Sa2, Sa3, Sa4, Saku, Yo}.
For results about $\tau$-cobordisms of fold maps with other set $\tau$ of singular fibers,
see \cite{Kal7, SaYa2}.

\section{Cobordism invariants of fold maps}\label{foldgerms}

\subsection{Fold germs and bundles of germs}

Let us define the fold germ \[g_{\la, q} \co (\R^{q+1},0) \to (\R,0)\] by 
\[  
g_{\la,q}(x_1,\ldots,x_{q+1})=(-x_1^2 - \cdots -x_{\la}^2 + x_{1+\la}^2 + \cdots + x_{1+q}^2)
\] 
for some $q \geq 1$ and $0 \leq \la \leq \lfloor (q+1)/2 \rfloor$.

We say that a pair of diffeomorphism germs
\[(\al \co (\R^{q+1},0) \to (\R^{q+1},0), \be \co (\R,0) \to (\R,0))\]
is an {\it automorphism} of a fold germ $g_{\la,q} \co (\R^{q+1},0) \to (\R,0)$
if the equation $g_{\la,q} \circ \al = \be \circ g_{\la,q}$ holds.
We will work with bundles whose fibers and structure groups are germs and  
groups of automorphisms of germs, respectively, see \cite{Jan}.

If we have a fold map $f \co Q^{n+q} \to N^n$, then for each 
$\la$ $(0 \leq \la \leq \lfloor (q+1)/2 \rfloor)$
we have a fold germ bundle $\csi_{\la}(f) \co E(\csi_{\la}(f)) \to S_{\la}(f)$ 
over the singular set of index $\la$ $S_{\la}(f)$, i.e., the
fiber of $\csi_{\la}(f)$ is the fold germ $g_{\la,q}$, and
over the singular set $S_{\la}(f)$ we have an $(\R^{q+1},0)$ bundle denoted by $\csi_{\la}^{q+1}(f) 
\co E(\csi_{\la}^{q+1}(f)) \to S_{\la}(f)$
and an  $(\R,0)$ bundle denoted by $\eta_{\la}^1(f)
\co E(\eta_{\la}^1(f)) \to S_{\la}(f)$ together with a fiberwise map 
$E(\csi_{\la}(f)) \co 
E(\csi_{\la}^{q+1}(f)) 
 \to  E(\eta_{\la}^1(f))$ which is equivalent on each fiber to 
 the fold germ $g_{\la,q}$.
 The base space of the fold germ bundle $\csi_{\la}(f)$ is the singular set of index $\la$ $S_{\la}(f)$
 and the total space of this bundle $\csi_{\la}(f)$ is the fiberwise map 
 $E(\csi_{\la}(f)) \co 
E(\csi_{\la}^{q+1}(f)) 
 \to  E(\eta_{\la}^1(f))$ between the total spaces of the bundles 
$\csi_{\la}^{q+1}(f)$ and $\eta_{\la}^1(f)$. 
We call the bundle $\eta_{\la}^1(f)$ the {\it target of the fold germ bundle 
$\csi_{\la}(f)$}.

 By \cite{Jan, Szucs3, Wa} this bundle $\csi_{\la}(f)$ is a locally trivial bundle 
 in a sense with a fiber $g_{\la,q}$
 and an appropriate group of automorphisms
\[(\al \co (\R^{q+1},0) \to (\R^{q+1},0), \be \co (\R,0) \to (\R,0))\] as structure group.
 By \cite{Jan, Wa}
this structure group can be reduced to a maximal compact subgroup, namely to the 
group $O(\la) \x O(q+1-\la)$ in the case of $0 \leq \la < (q+1)/2$ and the group 
generated by the group $O(\la) \x O(\la)$ and the transformation 
$ T = 
\begin{pmatrix}
0 & I_{\la} \\ I_{\la} & 0 
\end{pmatrix}
$ in the case of $\la = (q+1)/2$, see, for example, \cite{Sa1}.
We denote this latter group by $\langle O(\la) \x O(\la), T \rangle$.

It follows that the targets of the universal fold germ bundles of index 
$\la$ $(0 \leq \la \leq \lfloor (q+1)/2 \rfloor)$ 
are the trivial line bundles
$\eta_{\la,q}^1 \co \ep^1 \to B(O(\la) \x O(q+1-\la))$ for $\la \neq (q+1)/2$ and the appropriate line bundle
$\eta_{(q+1)/2,q}^1 \co l^1 \to B\langle O(\la) \x O(\la), T \rangle$ 
for $q$ odd.

For oriented manifolds $Q^{n+q}$ and $N^n$ we have the analogous statements if
we consider the subgroup $S(O(\la) \x O(q+1-\la))$ of orientation preserving transformations of the group
$O(\la) \x O(q+1-\la)$ 
and the trivial line bundle $\ep^1 \to BS(O(\la) \x O(q+1-\la))$
in the case of $0 \leq \la < (q+1)/2$,
and the appropriate subgroup $R_{\langle O(\la) \x O(\la), T \rangle}$
of the group $\langle O(\la) \x O(\la), T \rangle$
and the corresponding line bundle
$\tilde l^1 \to BR_{\langle O(\la) \x O(\la), T \rangle}$ 
in the case of $\la = (q+1)/2$.

\subsection{Immersions with prescribed normal bundles}

As an imitation of the method of lifting positive codimensional singular maps \cite{Szucs4} 
we can construct 
homomorphisms
\[
\csi_{{\la}, q}^N  \co \CC ob_{N,f}^{(O)}(n+q,-q) \to \imm^{\ep^1_{B(O(\la) \x O(q+1-\la))}}_N(n-1,1)
\]
for $0 \leq \la < (q+1)/2$ and
\[
\csi_{(q+1)/2, q}^N  \co \CC ob_{N,f}^{(O)}(n+q,-q) \to \imm^{l^1}_N(n-1,1)
\]
for $q$ odd
by mapping a cobordism class of a fold map $f$ into the cobordism class of the immersion of its 
fold singular set of index $\la$ $S_{\la}(f)$ with normal bundle induced from
the target of the universal fold germ bundle of index 
$\la$. (In the case of oriented manifolds $Q^{n+q}$ and $N^n$, we have the analogous homomorphisms 
\[
\csi_{{\la}, q}^{O,N}  \co \CC ob_{N,f}^{O}(n+q,-q) \to \imm^{\ep^1_{BS(O(\la) \x O(q+1-\la))}}_N(n-1,1)
\]
for $0 \leq \la < (q+1)/2$ and
\[
\csi_{(q+1)/2, q}^{O,N}  \co \CC ob_{N,f}^{O}(n+q,-q) \to \imm^{\tilde l^1}_N(n-1,1)
\]
for $q$ odd as well.)
We used these homomorphisms in \cite{Kal4, Kal5, Kal7} in order to describe cobordisms of fold maps. 

Since 
the cobordism group of $k$-codimensional immersions into a manifold
$N^{n}$ with normal bundle induced from a vector bundle $\csi^k$ is 
isomorphic to the group of stable homotopy classes $ \{ \dot N, T\csi^k \}$
\cite{We}, 
the homomorphisms $\csi_{{\la}, q}^N$ for $\la \neq (q+1)/2$ and $\csi_{(q+1)/2, q}^N$ for $q$ odd
can be considered as homomorphisms into the groups 
$\{ \dot N, T\ep^1_{B(O(\la) \x O(q+1-\la))} \}$ and $\{ \dot N, Tl^1 \}$,
respectively.
Without mentioning we identify the cobordism group of $k$-codimensional immersions into a manifold
$N^{n}$ with normal bundle induced from a vector bundle $\csi^k$ with 
the group of stable homotopy classes $ \{ \dot N, T\csi^k \}$.

We remark that the group $\{ \dot N, T\ep^1_{B(O(\la) \x O(q+1-\la))} \}$
is equal to the group 
$\{ \dot N, S^1 \vee SB(O(\la) \x O(q+1-\la)) \} \cong \{ \dot N, S^1 \} \oplus \{ \dot N, SB(O(\la) \x O(q+1-\la)) \}$,
where ``$S$'' denotes suspension.
Therefore the homomorphisms $\csi_{{\la}, q}^N$ ($\la \neq (q+1)/2$) can be written
in the forms 
\[
\csi_{{\la}, q, 1}^N \oplus \csi_{{\la}, q, 2}^N \co \CC ob_{N,f}^{(O)}(n+q,-q) \to 
\{ \dot N, S^1 \} \oplus \{ \dot N, SB(O(\la) \x O(q+1-\la)) \}
\] 
obviously.
Note that the homomorphism $\csi_{{\la}, q, 1}^N$ maps the fold cobordism class of
a fold map $f$ into the cobordism class of the framed immersion 
of the singular set of index $\la$ of the fold map $f$ $(0 \leq \la < (q+1)/2)$.

Note that $B(O(\la) \x O(q+1-\la)) = BO(\la) \x BO(q+1-\la)$ and there exists a composition of bundle maps
$\ep^1_{BO(q+1-\la)} \to \ep^1_{B(O(\la) \x O(q+1-\la))} \to \ep^1_{BO(q+1-\la)}$
which is the identity map. Therefore the group
$\{ \dot N, SBO(q+1-\la) \}$ is a direct summand of the group $\{ \dot N, SB(O(\la) \x O(q+1-\la)) \}$.

Let $\varrho_{{\la}, q}^N \co \imm^{\ep^1_{B(O(\la) \x O(q+1-\la))}}_N(n-1,1)
\to \imm^{\ep^1_{BO(q+1-\la)}}_N(n-1,1)$ denote the natural forgetting homomorphism.
Then we have weaker cobordism invariants 
\[
\varrho_{{\la}, q}^N \circ \csi_{{\la}, q}^N  \co \CC ob_{N,f}^{(O)}(n+q,-q) \to 
\{ \dot N, S^1 \} \oplus \{ \dot N, SBO(q+1-\la) \}
\]
($0 \leq \la < (q+1)/2$).

Let ${\tilde \theta}_q^N \co \imm^{l^1}_N(n-1,1) \to \imm_N(n-1,1)$ be
the natural forgetting homomorphism, where $\eta_{(q+1)/2,q}^1 \co l^1 \to B\langle O(\la) \x O(\la), T \rangle$
is the target of the universal fold germ bundle of index $(q+1)/2$ for $q$ odd.

A result about these invariants, which we obtain similarly to \cite{Kal4}, is the following.

\begin{theorem}\label{ori}
For $n \geq 1$\textup{,} an $n$-dimensional manifold $N^n$ and $q > 0$ 
the cobordism semigroup $\CC ob_{N, f}^{(O)}(n+q,-q)$ of 
fold maps of \textup{(}oriented\textup{)} $(n+q)$-dimensional manifolds into $N^n$
contains the direct sum of $\lfloor (q+1)/2 \rfloor$ copies of the group
$\{ \dot N, S^1 \}$\textup{.}
The restriction of the homomorphism $\csi_{{\la}, q}^N$ to the $(\la+1)$-th copy
of $\{ \dot N, S^1 \}$ is an isomorphism
$(\la = 0, \ldots,\lfloor (q-1)/2 \rfloor)$\textup{.}
\end{theorem}

\begin{theorem}\label{unori}
For $n \geq 1$\textup{,} an $n$-dimensional manifold $N^n$\textup{,} $q > 0$\textup{,}
$k \geq 1$ and
 $q = 2k -1$ the cobordism semigroup $\CC ob_{N, f}(n+q,-q)$ of 
fold maps of unoriented $(n+q)$-dimensional manifolds into $N^n$
contains the direct sum $\imm^{}_N(n-1,1) \oplus \bigoplus_{\la = 0}^{\lfloor (q-1)/2 \rfloor} \{ \dot N, S^1 \}$
as a direct summand\textup{.}
The direct summand $\imm^{}_N(n-1,1)$ is detected by the homomorphism
${\tilde \theta}_q^N \circ \csi_{(q+1)/2, q}^N \co \CC ob_{N,f}(n+q,-q) \to \imm^{}_N(n-1,1)$\textup{,} 
where ${\tilde \theta}_{q}^N \circ \csi_{(q+1)/2, q}^N$ maps 
a fold cobordism class $[f]$ to the cobordism class of the immersion of the
 singular set of index $k$ of the fold map $f$\textup{.} 
\end{theorem}

\begin{remark}
For $q$ even, in Theorems~\ref{ori} and \ref{unori} we could also chose the indeces
$\la = 1, \ldots,\lfloor (q+1)/2 \rfloor$ 
for the homomorphisms
$\csi_{{\la}, q, 1}^N$ instead of the indeces $\la = 0, \ldots,\lfloor (q-1)/2 \rfloor$.
The proof is similar to that of \cite{Kal4}, details are left to the reader.
\end{remark}

Another application of our invariants is the following result about simple fold maps,
which we obtained in \cite{Kal5}.

Let 
\[
\ga_n^N \co \imm^{\ep^1}_N(n-1,1) \oplus \imm^{\ep^1 \x \ga^1}_N(n-2,2) \to \imm_N(n-1,1) \oplus \imm^{\ga^1 \x \ga^1}_N(n-2,2)
\]
denote the natural forgetting homomorphism and
\[\phi_n^N \co \CC ob_{N,s}^O(n+1,-1) \to \CC ob_{N,f}^O(n+1,-1)\]
denote the natural homomorphism
which maps a simple fold cobordism class into its fold cobordism class.

When the codimension is equal to $-1$ and the target manifold $N^n$ is
oriented, in \cite{Kal5} we defined a semigroup homomorphism
\[
{\mathcal I_N} \co \CC ob_{N,s}^O(n+1,-1) \to \imm^{\ep^1}_N(n-1,1) \oplus \imm^{\ep^1 \x \ga^1}_N(n-2,2),
\]
which is just an adaptation of our invariant $\csi_{1, 1}^N$ to 
the case of simple fold maps of oriented manifolds into oriented manifolds and their oriented simple fold cobordisms.

In \cite{Kal5} we showed that 
there exists a homomorphism
\[
\theta_n^N \co \imm^{{\mathrm {det}}(\ga^1 \x \ga^1)}_N(n-1,1)
\to \imm_N(n-1,1) \oplus \imm^{\ga^1 \x \ga^1}_N(n-2,2)
\]
such that the diagram 
\begin{equation}\label{invaridiag}
\begin{CD}
\CC ob_{N,s}^O(n+1,-1) @> {\mathcal I_N} >> \imm^{\ep^1}_N(n-1,1) \oplus \imm^{\ep^1 \x \ga^1}_N(n-2,2) \\
@VV \phi_n^N V @V \ga_n^N VV \\
\CC ob_{N,f}^O(n+1,-1) @> \theta_n^N \circ \csi_{1, 1}^N >> \imm^{}_N(n-1,1) \oplus \imm^{\ga^1 \x \ga^1}_N(n-2,2).
\end{CD}
\end{equation}
commutes and we obtained the following.

\begin{theorem}
Let $N^n$ be an oriented manifold\textup{.} Then\textup{,}
the semigroup homomorphism $\mathcal I_N$ is a semigroup isomorphism between the 
cobordism semigroup $\CC ob_{N,s}^O(n+1,-1)$ of simple fold maps and the group 
$\imm^{\ep^1}_N(n-1,1) \oplus \imm^{\ep^1 \x \ga^1}_N(n-2,2)$\textup{.}
\end{theorem}

Let
\[
\ga_{n,1}^N \co \imm^{\ep^1}_N(n-1,1)  \to \imm^{}_N(n-1,1) 
\]
and
\[
\ga_{n,2}^N \co \imm^{\ep^1 \x \ga^1}_N(n-2,2) \to \imm^{\ga^1 \x \ga^1}_N(n-2,2).
\]
denote the natural forgetting homomorphisms.

Let $\pi_{n,2}^N \co \CC ob_{N,s}^O(n+1,-1) \to \imm^{\ep^1 \x \ga^1}_N(n-2,2)$
denote the composition of $\mathcal I_N$ with the projection to the second factor.

\begin{theorem}\label{injdontes}
If two simple fold cobordism classes $[f]$ and $[g]$ in $\CC ob_{N,s}^O(n+1,-1)$
are mapped into distinct elements by the natural homomorphism
$\ga_{n,2}^N \circ \pi_{n,2}^N$\textup{,}
then $[f]$ and $[g]$ are not fold cobordant\textup{.}
If $\ga_{n,2}^N$ is injective\textup{,} then so is $\phi_n^N$\textup{.}

If there exists a fold map from a not null-cobordant 
$(n+1)$-dimensional manifold into $N^n$\textup{,} then
$\phi_n^N$ is not surjective\textup{.}
\end{theorem}

\subsection{Pontryagin-Thom type construction}

In \cite{Kal6} among others we show the following, which is
a negative codimensional analogue of the Pontryagin-Thom type construction for singular maps in positive codimension
\cite{RSz, Szucs1, Sz1, Sz2, Szucs2, Sz3}.

\begin{theorem}\label{PTcons}
There is a Pontryagin-Thom type construction for $-1$ codimensional fold maps\textup{,} i\textup{.}e\textup{.,} 
\begin{enumerate}[\textup{(}\rm 1\textup{)}]
\item
there exists a universal fold map 
$\csi_{-1}^{}  \co U_{-1} \to \Ga_{-1}$ such that\footnote{The spaces $U_{-1}$ and $\Ga_{-1}$
are not \textup{(}finite dimensional\textup{)} manifolds and so $\csi_{-1}^{}$ is not a fold map\textup{.}}
for every $-1$ codimensional fold map $g \co Q^{n+1} \to N^{n}$ there exists a commutative diagram 
\[
\begin{CD}
Q^{n+1} @>>> U_{-1} \\
@V g VV @V \csi_{-1}^{} VV \\
N^{n} @>>> \Ga_{-1}
\end{CD}
\]
moreover the arising map $N^{n} \to \Ga_{-1}$ is unique up to homotopy\textup{.} It will be denoted by $\chi_g$\textup{.}
The space $\Ga_{-1}$ is constructed by gluing together total spaces of vector bundles
corresponding to the possible fold singular fibers and their automorphisms\textup{.}
\item
For every positive integer $n$ and $n$-dimensional manifold $N^n$ 
there is a natural bijection 
\[
\chi_*^N \co \CC ob_{N^n,f}^{}(n+1,-1) \to [{\dot N}^n, \Ga_{-1}]
\]
between the set of fold cobordism classes $\CC ob_{N^n,f}^{}(n+1,-1)$ 
and the set of homotopy classes $[{\dot N}^n, \Ga_{-1}]$\textup{.} 
The map $\chi_*^N$ maps a fold cobordism class
$[g]$ into the homotopy class of the inducing map $\chi_g \co {\dot N}^{n} \to \Ga_{-1}$\textup{.}  
\end{enumerate}
\end{theorem}

By Theorem~\ref{PTcons} we have a bijective cobordism invariant 
$\chi_*^N \co \CC ob_{N^n,f}^{}(n+1,-1) \to [{\dot N}^n, \Ga_{-1}]$ which
is a group isomorphism $\chi_*^{\R^n} \co \CC ob_{f}^{}(n+1,-1) \to \pi_n(\Ga_{-1})$ in
the case of $N^n = \R^n$.

By defining the singular sets of index $0$ and $1$ of
the universal fold map \[\csi_{-1}^{}  \co U_{-1} \to \Ga_{-1}\]
in the obvious way and by inducing the immersions of these singular sets into the space $\Ga_{-1}$
we get two representatives of two stable homotopy classes
$\varrho_0 \in \{  \Ga_{-1}, T{\ep^1_{BO(2)}} \}$
and $\varrho_1 \in \{  \Ga_{-1}, T{l^1} \}$.

If we have a fold map $g \co Q^{n+1} \to N^n$, then we have 
the stable homotopy class $\chi_g^s$ of the inducing map 
$\chi_g \co {\dot N}^{n} \to \Ga_{-1}$ in the group $\{ {\dot N}^{n}, \Ga_{-1} \}$.
Hence we have the elements $\varrho_0 \circ \chi_g^s$ and $\varrho_1 \circ \chi_g^s$ in the groups
$\{ {\dot N}^{n}, T{\ep^1_{BO(2)}} \}$
and $\{  {\dot N}^{n}, T{l^1} \}$, respectively, which correspond to
the elements $\csi_{{0}, 1}^N([g])$ and $\csi_{{1}, 1}^N([g])$, respectively. 

Therefore we have the following.

\begin{prop}
$\csi_{{i}, 1}^N([g]) = \varrho_i \circ \chi_g^s$ for $i=0,1$\textup{.}
\end{prop}

\subsection{Cobordism class of the source manifold of a fold map}

We have a natural homomorphism 
$\si_{N, q}^{O} \co \CC ob_{N, f}^{O}(n+q,-q) \to \Omega_{n+q}$
which assigns to a class of a fold map $f \co Q^{n+q} \to N^n$ 
the cobordism class $[Q^{n+q}]$ of the source manifold $Q^{n+q}$.

It is an easy fact that $\si_{\R, q}^{O}$ is surjective
and the image of $\si_{\R^2, q}^{O}$ consists of 
the cobordism classes of $(2+q)$-dimensional manifolds with even Euler 
characteristic \cite{Lev1}. 

\begin{prop}\label{sourcezero}
Let $N^n$ be a stably parallelizable $n$-dimensional manifold\textup{,} where $n$ is even\textup{.}
Let $f \co Q^{n+1} \to N^n$ be a fold map of an orientable manifold $Q^{n+1}$ such that
its singular set $S_f$ is orientable\textup{.} 
Then\textup{,} the oriented cobordism class of the source manifold $Q^{n+1}$ is zero\textup{.}
\end{prop}
\begin{proof}
By \cite[Lemma 3.1.]{Sa1} the bundle $TQ^{n+1} \oplus \ep^1$  is isomorphic to
the bundle $f^*TN^n \oplus \eta^{2}$ for some $2$-dimensional bundle $\eta^2$.
Hence the Stiefel-Whitney classes $w_j(Q^{n+1})$ are zero for $j \geq 3$ and
the Pontryagin classes $p_j(Q^{n+1})$ are zero for $j \geq 2$.
Since the Stiefel-Whitney class $w_1(Q^{n+1})$ is zero and 
$n+1$ is odd, every Stiefel-Whitney and Pontryagin numbers of the manifold $Q^{n+1}$ are zero.
This completes the proof.
\end{proof}

\begin{remark}
Proposition~\ref{sourcezero} generalizes the analogous result about
simple fold maps \cite{Kal5, Sa1}. 
\end{remark}

\begin{prop}\label{rankbecsles}
Let $q$ be even and let $N^n$ be a stably parallelizable manifold\textup{.}
Then\textup{,} the rank of the image of $\si_{N, q}^{O}$ is less than or equal to
the number of partitions of $(n+q)/4$ where each number in a partition is
less than or equal to $(q+1)/2$\textup{.}
In particular\textup{,} if $n > q+2$\textup{,} then
the homomorphism 
\[
\si_{N, q}^{O}\otimes \Q \co \CC ob_{N, f}^{O}(n+q,-q) \otimes \Q \to \Omega_{n+q} \otimes \Q
\]
is not surjective\textup{.} 
\end{prop}
\begin{proof}
The statement follows from the fact that
the possible non-zero Pontryagin numbers of the source manifold of
a representative $f \co Q^{n+q} \to N^n$ of an element in  $\CC ob_{N, f}^{O}(n+q,-q)$
are in bijection with the partitions of $(n+q)/4$ where each number in a partition is
less than or equal to $(q+1)/2$. This holds because by 
\cite[Lemma 3.1.]{Sa1} the bundle $TQ^{n+q} \oplus \ep^1$  is isomorphic to
the bundle $f^*TN^n \oplus \eta^{q+1}$ for some $(q+1)$-dimensional bundle $\eta^{q+1}$.
Hence the Pontryagin classes $p_i(Q^{n+q})$ are zero for $i > (q+1)/2$.
\end{proof}

\begin{cor}
Let $N^n$ be a stably parallelizable manifold\textup{.}
\begin{enumerate}[\textup{(}\rm 1\textup{)}]
\item
The orientable $(n+2)$-dimensional manifolds which have fold map into $N^n$ 
generate a subgroup of rank at most $1$ in the cobordism group of $(n+2)$-dimensional manifolds\textup{.}
\item
Let $n = 4k-2$\textup{.}
Let $Q^{4k}$ be a $(4k)$-dimensional oriented manifold which has a fold map into the 
stably parallelizable manifold $N^{4k-2}$\textup{.}
Then\textup{,} the signature $\si(Q^{4k})$ of $Q^{4k}$ is equal to
$\frac{2^{2k} B_k}{(2k)!} (-1)^{k+1} \langle p_1^k(Q^{4k}), [Q^{4k}] \rangle$\textup{,} where
$B_k$ denotes the $k${\rm th} Bernoulli number\textup{.}
\item
Let $n = 4k-1$\textup{.}
If $Q^{4k}$ has a fold map into $N^{4k-1}$ such that the singular set $S_f$ is orientable\textup{,} then
the same holds for the signature of $Q^{4k}$ as above\textup{.}
\end{enumerate}
\end{cor}
\begin{proof}
(1) follows from Proposition~\ref{rankbecsles}. (2) and (3) follow from the 
Hirzebruch signature theorem (see, for example, \cite[Lemma 19.1. and Theorem 19.4.]{MiSta}) and 
the fact that the Pontryagin classes $p_i(Q^{4k})$ are zero for $i > 1$.
\end{proof}

For other results about the signatures of source manifolds of fold maps, see, for example, \cite{Sa5, SaYa}.

\section{Subgroups of the cobordism (semi)group of fold maps}\label{subgroups}

In this section we extend the results of Theorems~\ref{ori} and \ref{unori}.

Let $O(1,k)$ denote the subgroup of the orthogonal group $O(k+1)$ 
whose elements are of the form $\begin{pmatrix}
1 & 0 \\ 0 & M 
\end{pmatrix}$
where $M$ is an element of the group $O(k)$.

\begin{theorem}
For $q > 1$\textup{,} the cobordism semigroup $\CC ob_{N,f}(n+q,-q)$ contains
the direct sum 
\[
\{ \dot N, S^1 \} \oplus \{ \dot N, SB(O(1) \x O(q)) \} \oplus
\bigoplus_{2 \leq \la < (q+1)/2} 
\{ \dot N, S^1 \} \oplus \{ \dot N, 
SBO(q+1-\la) \}
\]
as a direct summand\textup{.} 
\end{theorem}

\begin{proof}
Similarly to \cite{Kal4} we construct fibrations of Morse functions of $(q+1)$-spheres over
the framed immersions in $\imm^{\ep^1_{BO(q+1-\la)}}_N(n-1,1) \cong \{ \dot N, S^1 \} 
\oplus \{ \dot N, SBO(q+1-\la) \}$ with the appropriate orthogonal group as symmetry group.

For $1 \leq j < (q+1)/2$, 
let $h_j \co S^{q+1} \to (-\ep,\ep)$ be a Morse function of the $(q+1)$-dimensional sphere into the
open interval $(-\ep,\ep)$
with four critical points of index 
$0, j-1, j, q+1$, respectively, such that the critical value of index $j$ is zero in the interval $(-\ep,\ep)$.

In the following lemma we are looking for automorphisms of the Morse function $h_j$, i.e., 
auto-diffeomorphisms of the $(q+1)$-sphere which commute with the function $h_j$ $(1 \leq j < (q+1)/2)$.

\begin{lemma}\label{kiterj}
Let $1 \leq j < (q+1)/2$\textup{.}
There exists an identification of the 
Morse function $h_j$ around its critical point of index $j$
with the fold germ
\[
g_{j,q}(x_1,\ldots,x_{q+1})=(-x_1^2 - \cdots -x_{j}^2 + x_{1+j}^2 + \cdots + x_{1+q}^2),
\]
such that under this identification
\begin{enumerate}[\textup{(}\rm 1\textup{)}]
\item
any automorphism in the automorphism group $O(1) \x O(q)$ \textup{(}in the case of $j=1$\textup{)}
\item
any automorphism in the automorphism group 
$O(1,j-1) \x O(q+1-j)$ 
\textup{(}in the case of $j>1$\textup{)} 
\end{enumerate}
of the fold germ $g_{j,q}$
can be extended to an automorphism of the Morse function $h_j$\textup{.}
\end{lemma}
\begin{proof}
Let $1 \leq j < (q+1)/2$. Recall that the Morse function $h_j$ determines a 
handlebody decomposition of the $(q+1)$-sphere in the well-known way.
We give a model of the $(q+1)$-sphere, its Morse function $h_j$ and its handlebody decomposition 
which shows the statement of the lemma.

First let us suppose that $j>1$. 
Let us identify the neighbourhood of the critical point of index $0$ of $h_j$
with the hemisphere $D^{q+1} \subset S^{q+1}$ whose center is the critical point of index $0$.
Let us identify the interior of the other hemisphere with the Euclidean space $\R^{q+1}$.
We can choose this identification so that for the critical points of $h_j$ and
the corresponding handles of the handlebody decomposition of the $(q+1)$-sphere we have the following.
\begin{enumerate}
\item
The critical point of index $j-1$ of $h_j$ corresponds to 
the point $(-1, 0, \ldots, 0)$ in $\R^{q+1}$, 
\item
the $(j-1)$-handle
corresponds to a $(j-1)$-dimensional affine subspace $L^{j-1}$ in $\R^{q+1}$ which
is orthogonal to the line $\langle x_1 \rangle$
and contains the point $(-1, 0, \ldots, 0)$,
\item
the critical point of index $j$ of $h_j$
corresponds to the origin in $\R^{q+1}$,
\item
the $j$-handle corresponds to the half space $H_+ = \{ (x_1,\ldots,x_{q+1}) : x_1 \geq -1 \}$,
\item
the critical point of index $q+1$ of $h_j$ corresponds to 
the point $(-2, 0, \ldots, 0)$ in $\R^{q+1}$ 
and 
\item
the $(q+1)$-handle corresponds to the half space $H_- = \{ (x_1,\ldots,x_{q+1}) : x_1 \leq -1 \}$.
\end{enumerate}

Now we can see that the negative definite part of the fold germ $g_{j}$ 
corresponds to $J = (L^{j-1} \oplus \langle x_1 \rangle) \cap H_+$
and the positive definite part corresponds to the 
orthogonal complement of $J$, which lies in the
hyperplane $H = \{ x_1 = 0 \}$.
It is easy to check that an orthogonal transformation $O$ of $\R^{q+1}$ can be extended
to an automorphism of the Morse function $h_j$ in this model if and only if
the hyperplane $H$ and the affine subspace $L^{j-1}$ are invariant subspaces 
of the transformation $O$ and the line $\langle x_1 \rangle$ is fixed pointwise.
The statement of the lemma follows.

In the case of $j=1$ using the handle decomposition of the $(q+1)$-sphere determined by the Morse function
$h_1$ 
it is easy to check that the claimed symmetries exist.
\end{proof}

Let us define homomorphisms 
$\al_{1,q}^N \co \imm^{\ep^1_{B(O(1) \x O(q))}}_N(n-1,1) \to \CC ob_{N,f}(n+q,-q)$
and $\al_{j,q}^N \co \imm^{\ep^1_{B(O(1,j-1) \x O(q+1-j))}}_N(n-1,1) \to \CC ob_{N,f}(n+q,-q)$
 $(2 \leq j < (q+1)/2)$ similarly to \cite{Kal4} as follows.
 Let $[g \co M^{n-1} \to N^{n}]$ be an element of $\imm^{\ep^1_{B(O(1) \x O(q))}}_N(n-1,1)$.
Then the normal bundle of the immersion $g \co M^{n-1} \to N^{n}$ is 
induced from the trivial line bundle $\ep^1_{B(O(1) \x O(q))}$.
Let $\al_{1,q}^N([g])$ be 
 the cobordism class of the fold map which we obtain by constructing the fibration
 of the Morse function $h_1$ over the manifold $M^{n-1}$ with the same symmetry group as
that of the normal bundle of $g \co M^{n-1} \to N^{n}$ induced from the bundle $\ep^1_{B(O(1) \x O(q))}$,
in a way analogous to the method described in \cite{Kal4}.
 By Lemma~\ref{kiterj} the homomorphism $\al_{1,q}^N$ is well-defined.
 The definition of the homomorphisms $\al_{j,q}^N$ for $j>1$ is analogous.  
  
 Now we can see that the composition
\[
\csi_{{1}, q}^N \circ \al_{1,q}^N \co  \imm^{\ep^1_{B(O(1) \x O(q))}}_N(n-1,1)
\to \imm^{\ep^1_{B(O(1) \x O(q))}}_N(n-1,1)
\]
is the identity map and
the composition
\[
\varrho_{{j}, q}^N \circ 
\csi_{{j}, q}^N \circ \al_{j,q}^N \co  \imm^{\ep^1_{B(O(1,j-1) \x O(q+1-j))}}_N(n-1,1)
\to \imm^{\ep^1_{BO(q+1-j)}}_N(n-1,1)
\]
restricted to the direct summand $\imm^{\ep^1_{BO(q+1-j)}}_N(n-1,1)$ of
\[\imm^{\ep^1_{B(O(1,j-1) \x O(q+1-j))}}_N(n-1,1)\]
is the identity map as well $(j < (q+1)/2)$.

The remaining steps of the proof are obvious analogues of \cite{Kal4}.
\end{proof}

\begin{remark}
It follows that the composition
\[
\csi_{{j}, q}^N \circ \al_{j,q}^N \co \imm^{\ep^1_{B(O(1,j-1) \x O(q+1-j))}}_N(n-1,1) \to
\imm^{\ep^1_{B(O(j) \x O(q+1-j))}}_N(n-1,1)
\]
 is equal to the natural homomorphism 
\begin{multline*}
\be_{j,*} \co
\{ \dot N, S^1 \} \oplus \{ \dot N, SB(O(1,j-1) \x O(q+1-j)) \} 
\longrightarrow \\ \{ \dot N, S^1 \} \oplus \{ \dot N, SB(O(j) \x O(q+1-j)) \} 
\end{multline*}
induced by the map $\be_j \co BO(1,j-1) \to BO(j)$
$(2 \leq j < (q+1)/2)$.
Therefore if the map $\be_{j,*}$ is injective or an isomorphism, then the
cobordism semigroup $\CC ob_{N,f}(n+q,-q)$ contains
the group $\{ \dot N, SB(O(1,j-1) \x O(q+1-j)) \}$ as a subgroup or as a direct summand, respectively.

For example, in the case of $n=2$ and $N^2 = \R^2$, we have that
the cobordism group $\CC ob_{f}(2+q,-q)$ contains the direct sum
\[
\bigoplus_{1 \leq j < (q+1)/2} \pi_1^s \oplus \pi_1^s (B(O(1,j-1) \x O(q+1-j)))
= \Z_2^{3\lfloor q/2 \rfloor}
\]
as a direct summand, where $O(1,0)$ denotes the orthogonal group $O(1)$.
Similarly, the cobordism group $\CC ob_{f}^O(2+q,-q)$
contains the group $\Z_2^{2\lfloor q/2 \rfloor}$ as a direct summand.
\end{remark}

\end{document}